\RequirePackage{fix-cm}
\documentclass[11pt]{article}
\usepackage[T1]{fontenc}
\usepackage[utf8]{inputenc}
\usepackage{amsmath,amssymb,amsthm,mathtools,mathrsfs,bm}
\usepackage{enumitem}
\usepackage{microtype}
\usepackage[margin=1in]{geometry}
\usepackage{hyperref}
\usepackage[nameinlink,capitalize]{cleveref}
\hypersetup{colorlinks=true,linkcolor=blue,citecolor=blue,urlcolor=blue}

\theoremstyle{plain}
\newtheorem{theorem}{Theorem}[section]
\newtheorem{proposition}[theorem]{Proposition}
\newtheorem{lemma}[theorem]{Lemma}
\newtheorem{corollary}[theorem]{Corollary}
\theoremstyle{definition}
\newtheorem{definition}[theorem]{Definition}
\theoremstyle{remark}
\newtheorem{remark}[theorem]{Remark}

\begin{document}
\title{Exact Traces for Vector-Valued Morrey Spaces}

\author{Rishad Shahmurov\thanks{Cellular Products Research and Development, Roswell, GA, USA. E-mail: rshahmurov@crimson.ua.edu} \and Veli Shahmurov\thanks{Corresponding author. Antalya Bilim University, Antalya, Turkey. E-mail: veli.sahmurov@antalya.edu.tr}}

\date{}
\maketitle

\begin{abstract}
We establish an exact trace theorem for vector-valued time-Morrey spaces and complementary continuity theorems for time-BMO evolution classes. For a Banach couple $X_1\hookrightarrow X_0$, $1<p<\infty$ and $0<\lambda<1$, we prove that the trace of the class $u\in\mathcal M^{p,\lambda}(X_1)$, $u_t\in\mathcal M^{p,\lambda}(X_0)$ is exactly $(X_0,X_1)_{\theta,\infty}$, where $\theta=1-(1-\lambda)/p$. A constructive extension realizes every admissible trace, and counterexamples establish optimality of both the smoothness exponent and the weak fine index. Time-BMO control yields continuous finite-index interpolation traces and, under a mixed boundedness hypothesis, a logarithmic time modulus whose base-space logarithm is necessary. For the quadratic viscous Hamilton--Jacobi equation on the torus, we characterize the bounded initial data generating Morrey-regular solutions by the exact Besov condition $u_0\in B_{2,\infty}^{2\theta}$. The proof uses Hopf--Cole linearization and a direct composition argument, and requires no smallness assumption. Lacunary nonlinear solutions realize the sharp trace threshold. A forced Hamilton--Jacobi example exhibits genuinely logarithmic time regularity, and a forward--backward mean-field-game system illustrates initial and terminal trace compatibility.
\end{abstract}

\section{Introduction}
Which initial values are compatible with an evolution whose regularity is controlled on every time interval? We answer this question for vector-valued Morrey classes and establish complementary trace-continuity results under bounded-mean-oscillation (BMO) control in time. The main conclusions concern the exact trace range, optimal interpolation parameters, quantitative continuity, and the initial data of a nonlinear parabolic equation.

The classical Lions--Peetre theorem identifies the trace of
\[
W^{1,p}(0,T;X_0)\cap L^p(0,T;X_1),\qquad X_1\hookrightarrow X_0,
\]
with $(X_0,X_1)_{1-1/p,p}$ and provides norm-controlled extensions~\cite{LionsPeetre1964,Peetre1966}. This connection between time regularity and admissible initial values is central to abstract parabolic theory and maximal regularity~\cite{Weis2001,KunstmannWeis2004}.

\paragraph{Exact Morrey traces and sharpness.}
For $1<p<\infty$, $0<\lambda<1$ and
$\theta=1-(1-\lambda)/p$, let $\mathbb E^{p,\lambda}$ consist of the functions with
$u\in\mathcal M^{p,\lambda}(0,T;X_1)$ and
$u_t\in\mathcal M^{p,\lambda}(0,T;X_0)$. We prove
\[
 \operatorname{Tr}_0\mathbb E^{p,\lambda}
       =(X_0,X_1)_{\theta,\infty}.
\]
Every element of this space has a norm-controlled dyadic extension. The trace is bounded at every time, and the solution has an $X_0$-valued H\"older representative of exponent $\theta$. Both interpolation parameters are optimal: the smoothness exponent cannot be increased, and no finite fine index can replace $\infty$. Thus the passage from $L^p$ to time-Morrey control changes both the trace smoothness and the fine interpolation scale. The limiting cases $\lambda=0$ and $\lambda=1$ have distinct conclusions and are treated accordingly.

\paragraph{An exact nonlinear initial-data criterion.}
Consider the quadratic viscous Hamilton--Jacobi equation
\begin{equation}\label{eq:hj-intro}
 \begin{cases}
 u_t-\nu\Delta u+\tfrac12|\nabla u|^2=0,
       &(t,x)\in(0,T)\times\mathbb T^d,\\
 u(0,x)=u_0(x),&x\in\mathbb T^d,
 \end{cases}
\end{equation}
where $\nu>0$. Which bounded initial data produce a solution whose time derivative and second spatial derivatives satisfy time-Morrey estimates? We give the necessary-and-sufficient condition
\[
 \begin{split}
 u_0\in B_{2,\infty}^{2\theta}(\mathbb T^d)
 \quad\Longleftrightarrow\quad&
 u\in\mathcal M^{p,\lambda}(0,T;H^2(\mathbb T^d)),\\[-0.4ex]
 &u_t\in\mathcal M^{p,\lambda}(0,T;L^2(\mathbb T^d)),
 \end{split}
\]
for the bounded Hopf--Cole solution associated with any real
$u_0\in L^\infty(\mathbb T^d)$. The result holds in every spatial dimension and without a smallness condition, for fixed positive viscosity and finite $T$. Hopf--Cole linearization, heat-semigroup estimates, and a direct nonlinear composition argument prove sufficiency; the exact trace theorem proves necessity. Lacunary initial data realize both the optimal smoothness and the forced weak fine index within actual nonlinear solutions.

\paragraph{BMO traces and logarithmic continuity.}
If $u$ and $u_t$ have time-BMO control in $X_1$ and $X_0$, respectively, we prove that $u$ is continuous into $(X_0,X_1)_{1-1/q,q}$ for every $1<q<\infty$. Under the mixed assumptions $u\in L^\infty(0,T;X_1)$ and $u_t\in BMO_*(0,T;X_0)$, we obtain the quantitative modulus
\[
 \|u(t)-u(s)\|_{(X_0,X_1)_{\eta,\infty}}
 \lesssim
 \left(h\log\frac{eT}{h}\right)^{1-\eta}
 \|u_t\|_{BMO_*(X_0)}^{1-\eta}
 \|u\|_{L^\infty(X_1)}^\eta,
 \qquad h=|t-s|,
\]
for $0<\eta<1$. A scalar example proves that the logarithm cannot be removed from the base-space Lipschitz estimate. A forced Hamilton--Jacobi equation realizes this logarithmic behavior. We also apply the trace theorem to both ends of a forward--backward mean-field-game system, obtaining initial and terminal compatibility conditions for solutions with the prescribed Morrey regularity.

\paragraph{Relation to earlier work.}
Vector-valued trace theory in Sobolev, Besov, Triebel--Lizorkin, anisotropic and mixed-norm settings is developed in~\cite{Schneider2010,MeyriesVeraar2014,JohnsenSickel2018,AgrestiLindemulderVeraar2023,DenkRoodenburg2026}. Morrey interpolation and Besov--Morrey scales are studied in~\cite{BurenkovEtAl2010,BurenkovGhorbanalizadehSawano2019,Peetre1969Morrey,Sawano2018}, with parabolic applications in~\cite{OgawaSuguro2023}. Chill and Kr\'ol~\cite{ChillKrol2017} formulate a general Banach-function-space trace/$K$-method equivalence. The time-Morrey parameter norm in that framework agrees with the weak interpolation norm after a Hardy-operator estimate and an explicit norm identification; we record this relation after Theorem~\ref{thm:exact-morrey-trace-space}. Our finite-interval argument proves the trace characterization directly and supplies the extension, sharpness and continuity results used here.

The Hamilton--Jacobi criterion concerns Morrey regularity in the \emph{time} variable with spatial Sobolev values. It addresses a different regularity question from Lebesgue-data solution theory~\cite{BenArtziSoupletWeissler2002}, critical fractional-viscosity Besov theory~\cite{IwabuchiKawakami2017}, and spatial Besov--Morrey initial-data theory~\cite{Zhanpeisov2023}. Section~\ref{sec:hj-applications} proves the exact nonlinear criterion and the forced BMO estimate. The abstract BMO arguments combine the John--Nirenberg inequality~\cite{JohnNirenberg1961} with interpolation and time averaging. Throughout, the Morrey exact-range theorem, the BMO continuity conclusions, and the hypotheses needed for each PDE application are stated separately.
\section{Background and notation}\label{sec:background}
	
	\subsection{The \texorpdfstring{$K$}{K}-functional and real interpolation}
	
	Let $(X_0,X_1)$ be a compatible Banach couple, that is, $X_0$ and $X_1$ are continuously embedded into a common Hausdorff topological vector space. For $x\in X_0+X_1$ and $r>0$, the Peetre $K$-functional is defined by
	\[
	K(r,x;X_0,X_1)
	:=
	\inf_{x=x_0+x_1}
	\bigl(
	\|x_0\|_{X_0}+r\|x_1\|_{X_1}
	\bigr),
	\qquad x_0\in X_0,\ x_1\in X_1.
	\]
	
	We recall the basic properties used throughout the paper.
	
	\begin{proposition}\label{prop:K-basic}
		Let $(X_0,X_1)$ be a Banach couple, let $x\in X_0+X_1$, and let $r>0$. Then:
		
		\begin{enumerate}[label=\textup{(\roman*)}]
			\item $K(r,\cdot;X_0,X_1)$ is a norm on $X_0+X_1$, equivalent, for each fixed $r>0$, to the usual sum-space norm $K(1,\cdot;X_0,X_1)$;
			
			\item the map
			\[
			r\mapsto K(r,x;X_0,X_1)
			\]
			is nondecreasing on $(0,\infty)$;
			
			\item the map
			\[
			r\mapsto r^{-1}K(r,x;X_0,X_1)
			\]
			is nonincreasing on $(0,\infty)$;
			
			\item for every decomposition $x=x_0+x_1$ with $x_0\in X_0$ and $x_1\in X_1$,
			\[
			K(r,x;X_0,X_1)\le \|x_0\|_{X_0}+r\|x_1\|_{X_1};
			\]
			
			\item if $x\in X_0\cap X_1$, then
			\[
			K(r,x;X_0,X_1)\le \min\{\|x\|_{X_0},\,r\|x\|_{X_1}\}.
			\]
		\end{enumerate}
	\end{proposition}
	
	\begin{proof}
		All assertions are standard consequences of the definition. The monotonicity in \textup{(ii)} follows because increasing $r$ increases the quantity
		\[
		\|x_0\|_{X_0}+r\|x_1\|_{X_1}
		\]
		for every fixed decomposition. Assertion \textup{(iii)} is obtained by dividing the defining expression by $r$ and observing that
		\[
		r^{-1}K(r,x;X_0,X_1)
		=
		\inf_{x=x_0+x_1}
		\bigl(
		r^{-1}\|x_0\|_{X_0}+\|x_1\|_{X_1}
		\bigr),
		\]
		which is nonincreasing in $r$. The remaining statements are immediate.
	\end{proof}
	
	For $0<\theta<1$ and $1\le q\le\infty$, the real interpolation space $(X_0,X_1)_{\theta,q}$ consists of all $x\in X_0+X_1$ such that
	\[
	\|x\|_{(X_0,X_1)_{\theta,q}}
	:=
	\begin{cases}
		\left(
		\displaystyle\int_0^\infty
		\bigl(r^{-\theta}K(r,x;X_0,X_1)\bigr)^q
		\,\frac{dr}{r}
		\right)^{1/q},
		& 1\le q<\infty,\\[2ex]
		\displaystyle\sup_{r>0} r^{-\theta}K(r,x;X_0,X_1),
		& q=\infty,
	\end{cases}
	\]
	is finite.
	
	\begin{remark}\label{rem:index-infinity}
		The case $q=\infty$ is the natural target in the present paper. Indeed, Morrey-in-time estimates are based on a supremum over time intervals and therefore yield uniform bounds of the form
		\[
		\sup_{r>0} r^{-\theta}K(r,u(t);X_0,X_1)\le C,
		\]
		rather than integral bounds in the interpolation variable $r$. This is exactly the norm of the weak real interpolation space $(X_0,X_1)_{\theta,\infty}$.
	\end{remark}
	
	\subsection{Interpolation scales: positive examples and Morrey-type caveats}
	
	Many classical function spaces form interpolation scales under the real method. This is the case, in particular, for Lebesgue spaces, Sobolev spaces, Besov spaces, and large classes of Triebel--Lizorkin spaces. Trace theorems on these scales are central in elliptic and parabolic analysis.
	
	By contrast, one must be more careful with global Morrey spaces. Standard Morrey spaces do not, in general, form a real interpolation scale in the same clean way as the classical scales above. However, important positive results are available in several related settings:
	\begin{enumerate}[label=\textup{(\roman*)}]
		\item Sobolev--Morrey couples admit descriptions via Besov--Morrey scales;
		\item local Morrey-type spaces often have better interpolation behavior than global Morrey spaces;
		\item generalized Besov--Morrey and Triebel--Lizorkin--Morrey scales provide natural endpoint spaces for traces and embeddings.
	\end{enumerate}
	
	\begin{remark}\label{rem:morrey-scale-caveat}
		The abstract theorem proved below does not require Morrey spaces themselves to form an interpolation scale. We work with an arbitrary Banach couple $(X_0,X_1)$ and use only the general real-interpolation $K$-method. The Morrey structure enters solely through the time norm.
	\end{remark}
	
	\subsection{Time-Morrey spaces}
	
	\begin{definition}\label{def:time-morrey}
		Let $X$ be a Banach space, let $1<p<\infty$, let $0\le \lambda<1$, and let $T>0$. We define
		\[
		\|f\|_{\mathcal M^{p,\lambda}(0,T;X)}
		:=
		\sup_{I\subset(0,T)}
		|I|^{-\lambda/p}
		\left(\int_I \|f(t)\|_X^p\,dt\right)^{1/p},
		\]
		where the supremum is taken over all intervals $I\subset(0,T)$.
	\end{definition}
	
	\begin{definition}\label{def:holder-space}
		Let $X$ be a Banach space and let $0<\beta\le 1$. We write
		\[
		C^{0,\beta}([0,T];X)
		:=
		\Bigl\{u\in C([0,T];X):
		[u]_{C^{0,\beta}([0,T];X)}<\infty\Bigr\},
		\]
		where
		\[
		[u]_{C^{0,\beta}([0,T];X)}
		:=
		\sup_{0\le s<t\le T}
		\frac{\|u(t)-u(s)\|_X}{|t-s|^{\beta}}.
		\]
		The norm is
		\[
		\|u\|_{C^{0,\beta}([0,T];X)}
		:=
		\sup_{0\le t\le T}\|u(t)\|_X
		+
		[u]_{C^{0,\beta}([0,T];X)}.
		\]
		For $\beta=1$ this is the usual Lipschitz space. In particular, the statement
		\[
		u\in C^{0,\theta}([0,T];X_0)
		\]
		means that $u$ has a continuous $X_0$-valued representative and satisfies
		\[
		\|u(t)-u(s)\|_{X_0}\le C|t-s|^\theta
		\qquad(0\le s<t\le T).
		\]
	\end{definition}

	\begin{remark}\label{rem:lambda-limits}
		Let $X$ be a Banach space.
		
		\begin{enumerate}[label=\textup{(\roman*)}]
			\item If $\lambda=0$, then
			\[
			\mathcal M^{p,0}(0,T;X)=L^p(0,T;X)
			\]
			with equality of norms. In this regime,
			\[
			\theta=1-\frac1p.
			\]
			
			\item If $0<\lambda<1$, then the Morrey norm is stronger than the corresponding local $L^p$ control on short intervals and yields
			\[
			\int_I \|f(t)\|_{X}\,dt
			\le
			|I|^{1-\frac{1-\lambda}{p}}
			\|f\|_{\mathcal M^{p,\lambda}(0,T;X)}
			\]
			for every interval $I\subset(0,T)$.
			
			\item The trace exponent
			\[
			\theta=1-\frac{1-\lambda}{p}
			\]
			is strictly increasing in $\lambda$. As $\lambda\uparrow1$, one has $\theta\uparrow1$ and the trace space approaches the almost-strong endpoint
			\[
			(X_0,X_1)_{\theta,\infty},
			\qquad \theta\nearrow1.
			\]
			
			\item The endpoint $\lambda=1$ should not be described as a direct extension of Theorem~\ref{thm:morrey-trace-evolution}. By the Lebesgue differentiation theorem,
			\[
			\mathcal M^{p,1}(0,T;X)=L^\infty(0,T;X)
			\]
			with equivalent norms. Thus the endpoint evolution class becomes
			\[
			u\in W^{1,\infty}(0,T;X_0)\cap L^\infty(0,T;X_1).
			\]
			This gives Lipschitz continuity in $X_0$ and an essential $X_1$-bound, but it does not give strong continuity into $X_1$.
			
			\item The correct endpoint statement is therefore subcritical: for every $0<\eta<1$,
			\[
			W^{1,\infty}(0,T;X_0)\cap L^\infty(0,T;X_1)
			\hookrightarrow
			C^{0,1-\eta}([0,T];(X_0,X_1)_{\eta,\infty}).
			\]
			The endpoint $\eta=1$, which would correspond to $X_1$, generally fails without additional compactness, right-continuity, weak-continuity, or compatibility assumptions. This endpoint is treated separately in Subsection~\ref{subsec:endpoint-lambda-one}.
		\end{enumerate}
	\end{remark}

	\section{The Morrey trace theorem}\label{sec:trace}
	
	\begin{definition}\label{def:morrey-evolution-class}
		Let $(X_0,X_1)$ be a Banach couple with continuous embedding
		\[
		X_1\hookrightarrow X_0.
		\]
		Fix
		\[
		1<p<\infty,\qquad 0\le \lambda<1,\qquad T>0,
		\qquad
		\theta:=1-\frac{1-\lambda}{p}\in(0,1).
		\]
		We define
		\[
		\mathbb E^{p,\lambda}(0,T;X_0,X_1)
		:=
		\Bigl\{
		u:\;
		u\in \mathcal M^{p,\lambda}(0,T;X_1),\;
		\partial_t u\in \mathcal M^{p,\lambda}(0,T;X_0)
		\Bigr\},
		\]
		where $\partial_t u$ is understood in the sense of $X_0$-valued distributions.
	\end{definition}
	
	We now prove the Morrey trace theorem.

\begin{theorem}\label{thm:morrey-trace-evolution}
		Let $(X_0,X_1)$, $p$, $\lambda$, $T$, and $\theta$ be as in Definition~\ref{def:morrey-evolution-class}. Let
		\[
		u\in \mathbb E^{p,\lambda}(0,T;X_0,X_1).
		\]
		Then $u$ admits an $X_0$-valued absolutely continuous representative on $[0,T]$, still denoted by $u$, such that:
		
		\begin{enumerate}[label=\textup{(\roman*)}]
			\item for every $0\le s<t\le T$,
			\[
			u(t)-u(s)=\int_s^t \partial_t u(\tau)\,d\tau
			\qquad\text{in }X_0,
			\]
			and
			\[
			\|u(t)-u(s)\|_{X_0}
			\le
			|t-s|^\theta
			\|\partial_t u\|_{\mathcal M^{p,\lambda}(0,T;X_0)};
			\]
			
			\item for every $t\in[0,T]$,
			\[
			u(t)\in (X_0,X_1)_{\theta,\infty},
			\]
			and
			\[
			\sup_{t\in[0,T]}
			\|u(t)\|_{(X_0,X_1)_{\theta,\infty}}
			\le
			C\Bigl(
			\|u\|_{\mathcal M^{p,\lambda}(0,T;X_1)}
			+
			\|\partial_t u\|_{\mathcal M^{p,\lambda}(0,T;X_0)}
			\Bigr);
			\]
			
			\item in particular,
			\[
			u\in C^{0,\theta}([0,T];X_0)\cap L^\infty\bigl(0,T;(X_0,X_1)_{\theta,\infty}\bigr),
			\]
			and the endpoint traces
			\[
			u(0),\qquad u(T)
			\]
			are well-defined elements of $(X_0,X_1)_{\theta,\infty}$.
		\end{enumerate}
	\end{theorem}
	
	\begin{proof}
		We divide the proof into four steps.
		
		\medskip
		\noindent
		\textit{Step 1: Morrey control implies local $L^1$ control.}
		Let $X$ be a Banach space and let $f\in \mathcal M^{p,\lambda}(0,T;X)$. For any interval $I\subset(0,T)$, Hölder's inequality gives
		\[
		\int_I \|f(t)\|_X\,dt
		\le
		|I|^{1-\frac1p}
		\left(\int_I \|f(t)\|_X^p\,dt\right)^{1/p}.
		\]
		By Definition~\ref{def:time-morrey},
		\[
		\left(\int_I \|f(t)\|_X^p\,dt\right)^{1/p}
		\le
		|I|^{\lambda/p}\|f\|_{\mathcal M^{p,\lambda}(0,T;X)}.
		\]
		Hence
		\[
		\int_I \|f(t)\|_X\,dt
		\le
		|I|^{1-\frac{1-\lambda}{p}}
		\|f\|_{\mathcal M^{p,\lambda}(0,T;X)}
		=
		|I|^\theta
		\|f\|_{\mathcal M^{p,\lambda}(0,T;X)}.
		\]
		Therefore
		\[
		\mathcal M^{p,\lambda}(0,T;X)\hookrightarrow L^1(0,T;X)
		\]
		continuously on bounded intervals.
		
		Applying this with $X=X_0$ and using $X_1\hookrightarrow X_0$, we get
		\[
		u\in L^1(0,T;X_0),
		\qquad
		\partial_t u\in L^1(0,T;X_0).
		\]
		Thus
		\[
		u\in W^{1,1}(0,T;X_0),
		\]
		hence $u$ admits an absolutely continuous $X_0$-valued representative on $[0,T]$.
		
		\medskip
		\noindent
		\textit{Step 2: Hölder continuity in $X_0$.}
		Since $u\in W^{1,1}(0,T;X_0)$, for every $0\le s<t\le T$,
		\[
		u(t)-u(s)=\int_s^t \partial_t u(\tau)\,d\tau
		\qquad\text{in }X_0.
		\]
		Therefore
		\[
		\|u(t)-u(s)\|_{X_0}
		\le
		\int_s^t \|\partial_t u(\tau)\|_{X_0}\,d\tau
		\le
		|t-s|^\theta
		\|\partial_t u\|_{\mathcal M^{p,\lambda}(0,T;X_0)}.
		\]
		This proves \textup{(i)}.
		
		\medskip
		\noindent
		\textit{Step 3: Pointwise $K$-functional estimate.}
		Fix $t\in[0,T]$. For admissible $h>0$, define
		\[
		b_h(t):=\frac1h\int_t^{t+h}u(\tau)\,d\tau\in X_1,
		\qquad
		a_h(t):=u(t)-b_h(t)\in X_0,
		\]
		or the analogous backward average near $T$. Then
		\[
		u(t)=a_h(t)+b_h(t).
		\]
		
		Since
		\[
		u(\tau)-u(t)=\int_t^\tau \partial_t u(s)\,ds,
		\]
		we obtain
		\[
		\|a_h(t)\|_{X_0}
		\le
		h^\theta
		\|\partial_t u\|_{\mathcal M^{p,\lambda}(0,T;X_0)}.
		\]
		Also,
		\[
		\|b_h(t)\|_{X_1}
		\le
		h^{-(1-\theta)}
		\|u\|_{\mathcal M^{p,\lambda}(0,T;X_1)}.
		\]
		
		Thus, by Proposition~\ref{prop:K-basic},
		\[
		K(r,u(t);X_0,X_1)
		\le
		h^\theta A + r h^{-(1-\theta)} B,
		\]
		where
		\[
		A:=\|\partial_t u\|_{\mathcal M^{p,\lambda}(0,T;X_0)},
		\qquad
		B:=\|u\|_{\mathcal M^{p,\lambda}(0,T;X_1)}.
		\]
		We split the interpolation scales into two ranges.  If $0<r\le T/2$, then for every $t\in[0,T]$ at least one of the intervals $[t,t+r]$ and $[t-r,t]$ is contained in $[0,T]$.  Using the corresponding forward or backward average and choosing $h=r$ gives
		\[
		r^{-\theta}K(r,u(t);X_0,X_1)
		\le C(A+B),
		\qquad 0<r\le T/2,
		\]
		uniformly in $t$.

		For the remaining scales we use the $X_0$ part of the $K$-functional.  The Morrey bound on the full interval gives
		\[
		\left(\int_0^T\|u(s)\|_{X_1}^p\,ds\right)^{1/p}
		\le T^{\lambda/p}B.
		\]
		Hence there is a point $\tau\in(0,T)$ at which the $X_1$-valued representative is defined and
		\[
		\|u(\tau)\|_{X_1}
		\le T^{-(1-\lambda)/p}B
		=T^{\theta-1}B.
		\]
		Writing $c_{10}$ for the norm of the embedding $X_1\hookrightarrow X_0$, Step~2 yields
		\[
		\sup_{t\in[0,T]}\|u(t)\|_{X_0}
		\le c_{10}T^{\theta-1}B+T^\theta A.
		\]
		Therefore, if $r>T/2$,
		\[
		r^{-\theta}K(r,u(t);X_0,X_1)
		\le r^{-\theta}\|u(t)\|_{X_0}
		\le C_{T,c_{10}}(A+B).
		\]
		Combining the two scale ranges, we obtain
		\[
		\sup_{r>0}r^{-\theta}K(r,u(t);X_0,X_1)
		\le C_{p,\lambda,T,c_{10}}(A+B),
		\]
		uniformly in $t$.
		
		By the definition of the real interpolation norm,
		\[
		\|u(t)\|_{(X_0,X_1)_{\theta,\infty}}
		\le
		C(A+B).
		\]
		This proves \textup{(ii)}.
		
		\medskip
		\noindent
		\textit{Step 4: Endpoint traces.}
		Since $u\in C([0,T];X_0)$ by Step~2, the values $u(0)$ and $u(T)$ are well defined in $X_0$. Step~3 shows they belong in fact to $(X_0,X_1)_{\theta,\infty}$ and satisfy the same estimate. This proves \textup{(iii)}.
	\end{proof}
	
	\subsection{Exactness and sharpness of the Morrey trace space}\label{subsec:morrey-trace-sharpness}

	The preceding theorem gives the embedding direction of the Morrey trace theorem.  In the genuinely Morrey regime the target space is not only a convenient weak endpoint; it is the exact trace range.  The case \(\lambda=0\) is exceptional: then \(\mathcal M^{p,0}=L^p\), and the classical trace theorem identifies the sharper trace space \((X_0,X_1)_{1-1/p,p}\).  The theorem below therefore concerns \(0<\lambda<1\).

	We next identify the full trace range.

\begin{theorem}\label{thm:exact-morrey-trace-space}
	Let \((X_0,X_1)\) be a Banach couple with continuous embedding \(X_1\hookrightarrow X_0\).  Let
	\[
	1<p<\infty,\qquad 0<\lambda<1,\qquad T>0,
	\qquad
	\theta=1-\frac{1-\lambda}{p}.
	\]
	Then the trace at \(t=0\) of the Morrey evolution class is exactly
	\[
	\operatorname{Tr}_{0}\mathbb E^{p,\lambda}(0,T;X_0,X_1)
	=
	(X_0,X_1)_{\theta,\infty}.
	\]
	More precisely, the trace map
	\[
	\operatorname{Tr}_{0}:u\mapsto u(0)
	\]
	is bounded from \(\mathbb E^{p,\lambda}(0,T;X_0,X_1)\), equipped with its natural norm, into \((X_0,X_1)_{\theta,\infty}\).  Conversely, for every \(x\in (X_0,X_1)_{\theta,\infty}\) there exists
	\[
	u\in \mathbb E^{p,\lambda}(0,T;X_0,X_1)
	\]
	such that \(u(0)=x\) in \(X_0\), and
	\[
	\|u\|_{\mathcal M^{p,\lambda}(0,T;X_1)}
	+
	\|u'\|_{\mathcal M^{p,\lambda}(0,T;X_0)}
	\le
	C\|x\|_{(X_0,X_1)_{\theta,\infty}}.
	\]
	The constant \(C\) depends only on \(p,\lambda,T\), and the embedding constant of \(X_1\hookrightarrow X_0\).
	\end{theorem}

	\begin{proof}
	The boundedness of the trace map is exactly the trace estimate in Theorem~\ref{thm:morrey-trace-evolution}.  We prove the converse by an explicit dyadic extension.

	Let
	\[
	M:=\|x\|_{(X_0,X_1)_{\theta,\infty}}.
	\]
	For \(n=0,1,2,\ldots\), put
	\[
	r_n:=T2^{-n}.
	\]
	By the definition of the \((X_0,X_1)_{\theta,\infty}\)-norm, for each \(n\) we may choose a decomposition
	\[
	x=a_n+b_n,
	\qquad a_n\in X_0,
	\qquad b_n\in X_1,
	\]
	such that
	\[
	\|a_n\|_{X_0}+r_n\|b_n\|_{X_1}
	\le
	2K(r_n,x;X_0,X_1)
	\le
	2Mr_n^\theta.
	\]
	Therefore
	\[
	\|x-b_n\|_{X_0}=\|a_n\|_{X_0}\le 2Mr_n^\theta,
	\qquad
	\|b_n\|_{X_1}\le 2Mr_n^{\theta-1}.
	\]
	In particular \(b_n\to x\) in \(X_0\).

	Define \(u\) on each dyadic interval \((r_{n+1},r_n]\) by linear interpolation between \(b_{n+1}\) and \(b_n\):
	\[
	u(t):=
	\frac{t-r_{n+1}}{r_n-r_{n+1}}b_n
	+
	\frac{r_n-t}{r_n-r_{n+1}}b_{n+1},
	\qquad r_{n+1}<t\le r_n.
	\]
	Set \(u(0):=x\).  Since \(b_n\to x\) in \(X_0\), the resulting function is continuous at \(0\) as an \(X_0\)-valued function.  For \(t\in(r_{n+1},r_n]\), the preceding estimates give
	\[
	\|u(t)\|_{X_1}
	\le
	\max\{\|b_n\|_{X_1},\|b_{n+1}\|_{X_1}\}
	\le
	C M r_n^{\theta-1}.
	\]
	Since \(t\simeq r_n\) on \((r_{n+1},r_n]\), and \(\theta-1=-(1-\lambda)/p\), we have
	\[
	\|u(t)\|_{X_1}
	\le
	C M t^{-\frac{1-\lambda}{p}}.
	\]
	Similarly, on \((r_{n+1},r_n)\),
	\[
	u'(t)=\frac{b_n-b_{n+1}}{r_n-r_{n+1}}.
	\]
	Using \(b_n-x=-a_n\) and \(b_{n+1}-x=-a_{n+1}\),
	\[
	\|b_n-b_{n+1}\|_{X_0}
	\le
	\|a_n\|_{X_0}+\|a_{n+1}\|_{X_0}
	\le
	C M r_n^\theta.
	\]
	Since \(r_n-r_{n+1}=r_{n+1}\simeq r_n\), it follows that
	\[
	\|u'(t)\|_{X_0}
	\le
	C M r_n^{\theta-1}
	\le
	C M t^{-\frac{1-\lambda}{p}}.
	\]

	It remains only to observe that the scalar function
	\[
	g(t)=t^{-\frac{1-\lambda}{p}},\qquad 0<t<T,
	\]
	belongs to \(\mathcal M^{p,\lambda}(0,T)\) when \(0<\lambda<1\).  Indeed, if \(I=(a,b)\subset(0,T)\) and \(h=b-a\), then
	\[
	\int_I g(t)^p\,dt=\int_a^b t^{-(1-\lambda)}\,dt.
	\]
	If \(a\le h\), then this is bounded by \(C h^\lambda\).  If \(a>h\), then
	\[
	\int_a^b t^{-(1-\lambda)}\,dt
	\le
	h a^{-(1-\lambda)}
	\le
	h h^{-(1-\lambda)}=h^\lambda.
	\]
	Thus
	\[
	|I|^{-\lambda/p}
	\left(\int_I g(t)^p\,dt\right)^{1/p}
	\le C,
	\]
	uniformly in \(I\).  Hence \(u\in\mathcal M^{p,\lambda}(0,T;X_1)\) and \(u'\in\mathcal M^{p,\lambda}(0,T;X_0)\), with the asserted estimate. The piecewise derivative is integrable at zero because $\theta>0$. Telescoping across the dyadic nodes and passing to zero therefore gives $u(t)-x=\int_0^t u'(s)\,ds$ in $X_0$. In particular, no additional distributional derivative is created at the accumulation point.
	\end{proof}

	\paragraph{Relation with the general trace method.}
Chill--Kr\'ol's generalized $K$-method~\cite{ChillKrol2017} uses the norm
$\|t^{-1}K(t,x)\|_\Phi$. For
$\Phi=\mathcal M^{p,\lambda}(0,\infty)$ and $\theta=1-(1-\lambda)/p$, its relation to the trace norm identified above is
\begin{equation}\label{eq:general-method-specialization}
 \|x\|_{(X_0,X_1)_{\theta,\infty}}
 \le \|t^{-1}K(t,x)\|_\Phi
 \le \lambda^{-1/p}\|x\|_{(X_0,X_1)_{\theta,\infty}}.
\end{equation}
The lower estimate follows from monotonicity of $K(t,x)/t$. For the upper estimate, use
$K(t,x)/t\le t^{\theta-1}\|x\|_{(X_0,X_1)_{\theta,\infty}}$ and
$\int_b^{b+h}t^{\lambda-1}\,dt\le h^\lambda/\lambda$.
Moreover, dilation and Minkowski's inequality give
$\|P\|_{\mathcal L(\Phi)}\le\theta^{-1}$ for Hardy averaging
$Pf(t)=t^{-1}\int_0^t f(r)\,dr$.
These calculations connect the abstract trace-range descriptions.
The finite-interval proof above establishes the result and its extension estimate directly; the sharpness and endpoint continuity conclusions follow below.

\begin{proposition}\label{prop:sharp-morrey-holder-exponent}
	The exponent
	\[
	\theta=1-\frac{1-\lambda}{p}
	\]
	in the \(X_0\)-valued H\"older conclusion of Theorem~\ref{thm:morrey-trace-evolution} cannot be improved.
	\end{proposition}

	\begin{proof}
	For \(0<\lambda<1\), take \(X_0=X_1=\mathbb R\) and
	\[
	f(t)=t^{-\frac{1-\lambda}{p}},
	\qquad
	u(t)=\int_0^t f(s)\,ds=C_{p,\lambda}t^\theta.
	\]
	As shown in the proof of Theorem~\ref{thm:exact-morrey-trace-space}, \(f\in\mathcal M^{p,\lambda}(0,1)\).  Hence \(u'\in\mathcal M^{p,\lambda}(0,1)\), while \(u\) is not \(\beta\)-H\"older at \(0\) for any \(\beta>\theta\).

	For \(\lambda=0\), the borderline power \(t^{-1/p}\) is not in \(L^p(0,1)\).  A logarithmic correction gives the same conclusion.  Let
	\[
	f(t)=t^{-1/p}\bigl(\log(e/t)\bigr)^{-1},
	\qquad 0<t<1.
	\]
	Then \(f\in L^p(0,1)=\mathcal M^{p,0}(0,1)\), because
	\[
	\int_0^1 t^{-1}\bigl(\log(e/t)\bigr)^{-p}\,dt<\infty
	\]
	for \(p>1\).  Its primitive satisfies, up to harmless constants,
	\[
	u(t)\simeq t^{1-1/p}\bigl(\log(e/t)\bigr)^{-1}
	\]
	as \(t\downarrow0\).  Therefore \(u\) is not \(\beta\)-H\"older at \(0\) for any \(\beta>1-1/p\).  This proves sharpness also in the classical \(\lambda=0\) endpoint of the Morrey scale.
	\end{proof}

	\begin{proposition}\label{prop:q-infty-forced}
	Let \(0<\lambda<1\), \(1<p<\infty\), and \(\theta=1-(1-\lambda)/p\).  In general, the trace target in Theorem~\ref{thm:morrey-trace-evolution} cannot be replaced by \((X_0,X_1)_{\theta,q}\) for any finite \(q\).
	\end{proposition}

	\begin{proof}
	By Theorem~\ref{thm:exact-morrey-trace-space}, it suffices to exhibit a Banach couple and a vector which belongs to \((X_0,X_1)_{\theta,\infty}\) but not to \((X_0,X_1)_{\theta,q}\) for any finite \(q\).

	Let
	\[
	X_0=c_0,
	\qquad
	X_1=\Bigl\{x=(x_n)_{n\ge1}:\sup_{n\ge1}2^n|x_n|<\infty\Bigr\},
	\]
	with the natural norms
	\[
	\|x\|_{X_0}=\sup_n |x_n|,
	\qquad
	\|x\|_{X_1}=\sup_n 2^n|x_n|.
	\]
	Then \(X_1\hookrightarrow X_0\).  For this diagonal couple, the scalar coordinate calculation gives
	\[
	K(r,x;X_0,X_1)
	\simeq
	\sup_{n\ge1}\min\{1,r2^n\}|x_n|.
	\]
	Now take
	\[
	x_n=2^{-\theta n}.
	\]
	Then \(x\in c_0\), and
	\[
	\sup_{r>0}r^{-\theta}K(r,x;X_0,X_1)<\infty.
	\]
	Indeed, if \(2^{-(N+1)}<r\le2^{-N}\), then
	\[
	\sup_n \min\{1,r2^n\}2^{-\theta n}
	\lesssim r^\theta.
	\]
	Hence \(x\in (X_0,X_1)_{\theta,\infty}\).

	On the other hand, for the same \(r\)-range one also has the lower estimate
	\[
	K(r,x;X_0,X_1)
	\gtrsim r^\theta.
	\]
	Consequently
	\[
	\int_0^1\bigl(r^{-\theta}K(r,x;X_0,X_1)\bigr)^q\,\frac{dr}{r}
	=\infty
	\]
	for every finite \(q\).  Thus \(x\notin (X_0,X_1)_{\theta,q}\) for \(q<\infty\).  Since \(x\) is nevertheless the trace of a function in \(\mathbb E^{p,\lambda}\) by Theorem~\ref{thm:exact-morrey-trace-space}, no finite interpolation index can replace \(\infty\) in general. The same example also excludes every larger interpolation smoothness $\eta>\theta$: the lower bound yields $r^{-\eta}K(r,x)\gtrsim r^{\theta-\eta}\to\infty$ as $r\downarrow0$.
	\end{proof}

	\begin{remark}\label{rem:no-continuity-trace-space}
	The exact trace theorem identifies the range of the pointwise trace map.  It should not be read as a strong-continuity statement with values in \((X_0,X_1)_{\theta,\infty}\).  In general, Morrey control gives an \(X_0\)-continuous representative and uniformly bounded \((X_0,X_1)_{\theta,\infty}\)-traces.  Strong continuity in the weak interpolation norm requires additional hypotheses, such as compactness, stronger summability, or an independent continuity assumption in the interpolation space.
	\end{remark}

	\subsection{The endpoint \texorpdfstring{$\lambda=1$}{lambda=1} and BMO trace consequences}\label{subsec:endpoint-lambda-one}

	The endpoint $\lambda=1$ is not an endpoint of Theorem~\ref{thm:morrey-trace-evolution} in the literal sense $\theta=1$. It is a different theorem. With the normalization of Definition~\ref{def:time-morrey}, the endpoint Morrey space is $L^\infty$; hence the natural endpoint evolution class is
	\[
	\mathbb E^\infty(0,T;X_0,X_1)
	:=
	W^{1,\infty}(0,T;X_0)\cap L^\infty(0,T;X_1).
	\]
	The conclusion is continuity into every subcritical interpolation space, but not generally into $X_1$.

	\begin{lemma}\label{lem:morrey-endpoint-linfty}
	Let $X$ be a Banach space and $1<p<\infty$. If
	\[
	\|g\|_{\mathcal M^{p,1}(0,T;X)}
	:=
	\sup_{I\subset(0,T)} |I|^{-1/p}
	\left(\int_I\|g(t)\|_X^p\,dt\right)^{1/p},
	\]
	then
	\[
	\mathcal M^{p,1}(0,T;X)=L^\infty(0,T;X)
	\]
	with equivalent norms.
	\end{lemma}

	\begin{proof}
	The inclusion $L^\infty\subset\mathcal M^{p,1}$ is immediate. Conversely, the defining estimate gives uniform bounds for the $L^p$ averages of $\|g\|_X$ over all intervals. By the Lebesgue differentiation theorem applied to the scalar function $\|g(\cdot)\|_X^p$, these averages converge almost everywhere to $\|g(t)\|_X^p$, giving the $L^\infty$ bound.
	\end{proof}

	\begin{theorem}\label{thm:endpoint-subcritical-trace}
	Let $(X_0,X_1)$ be a Banach couple with $X_1\hookrightarrow X_0$. Assume
	\[
	u\in W^{1,\infty}(0,T;X_0)\cap L^\infty(0,T;X_1).
	\]
	Then, for every $0<\eta<1$,
	\[
	u\in C^{0,1-\eta}([0,T];(X_0,X_1)_{\eta,\infty}),
	\]
	and
	\begin{align*}
	&[u]_{C^{0,1-\eta}([0,T];(X_0,X_1)_{\eta,\infty})}
	\le
	C_\eta
	\|u'\|_{L^\infty(0,T;X_0)}^{1-\eta}
	\|u\|_{L^\infty(0,T;X_1)}^{\eta}.
	\end{align*}
	Moreover,
	\[
	\sup_{0\le t\le T}\|u(t)\|_{(X_0,X_1)_{\eta,\infty}}
	\le
	C_\eta
	\|u\|_{L^\infty(0,T;X_0)}^{1-\eta}
	\|u\|_{L^\infty(0,T;X_1)}^\eta.
	\]
	\end{theorem}

	\begin{proof}
	We first record the elementary interpolation inequality
	\[
	\|x\|_{(X_0,X_1)_{\eta,\infty}}
	\le
	C_\eta \|x\|_{X_0}^{1-\eta}\|x\|_{X_1}^{\eta},
	\qquad x\in X_0\cap X_1.
	\]
	Indeed,
	\[
	K(r,x;X_0,X_1)\le \min\{\|x\|_{X_0},r\|x\|_{X_1}\},
	\]
	and taking the supremum of $r^{-\eta}K(r,x)$ over $r>0$ gives the estimate.

	For almost every $s,t$ we may apply this to $x=u(t)-u(s)$. Since
	\[
	\|u(t)-u(s)\|_{X_0}
	\le |t-s|\|u'\|_{L^\infty(0,T;X_0)}
	\]
	and
	\[
	\|u(t)-u(s)\|_{X_1}
	\le 2\|u\|_{L^\infty(0,T;X_1)},
	\]
	we obtain
	\[
	\|u(t)-u(s)\|_{(X_0,X_1)_{\eta,\infty}}
	\le
	C_\eta |t-s|^{1-\eta}
	\|u'\|_{L^\infty(X_0)}^{1-\eta}
	\|u\|_{L^\infty(X_1)}^{\eta}.
	\]
	The estimate extends to every pair $s,t\in[0,T]$ by taking the $X_0$-continuous representative and approximating by Lebesgue points of the $X_1$-valued representative. The same interpolation inequality applied to $x=u(t)$ gives the uniform trace bound. This proves the theorem.
	\end{proof}

	\begin{remark}\label{rem:endpoint-no-x1}
	Theorem~\ref{thm:endpoint-subcritical-trace} reaches every subcritical space $(X_0,X_1)_{\eta,\infty}$, $0<\eta<1$. It does not generally reach $X_1$. The available endpoint information is only
	\[
	u\in L^\infty(0,T;X_1),
	\qquad
	u\in C^{0,1}([0,T];X_0),
	\]
	and these two facts do not imply $u\in C([0,T];X_1)$. Hence $\lambda=1$ is included through a different endpoint theorem rather than by inserting $\theta=1$ into the Morrey trace theorem.
	\end{remark}

	\subsubsection{BMO endpoint trace consequences}\label{subsubsec:bmo-trace}

	The endpoint of Calder\'on--Zygmund theory is naturally expressed in terms of \(BMO\).  The following trace consequences are therefore useful when maximal-regularity estimates give \(Au\) and \(u'\) in \(BMO\), rather than in \(L^\infty\).  We use the finite-interval norm
	\[
	\|g\|_{BMO_*(0,T;X)}:=\|g_{(0,T)}\|_X+[g]_{BMO(0,T;X)},
	\]
	where
	\[
	g_I:=\frac1{|I|}\int_I g(t)\,dt,
	\qquad
	[g]_{BMO(0,T;X)}:=
	\sup_{I\subset(0,T)}\frac1{|I|}\int_I\|g(t)-g_I\|_X\,dt.
	\]
	The average term is included because \(BMO\) itself is only a seminorm modulo constants.

	\begin{lemma}\label{lem:bmo-to-lq}
	Let \(X\) be a Banach space and \(1<q<\infty\).  Then
	\[
	BMO_*(0,T;X)\hookrightarrow L^q(0,T;X),
	\]
	and
	\[
	\|g\|_{L^q(0,T;X)}
	\le
	C_q T^{1/q}\|g\|_{BMO_*(0,T;X)}.
	\]
	\end{lemma}

	\begin{proof}
	Let \(J=(0,T)\) and write \(g_J\) for the average of \(g\) on \(J\).  Set
	\[
	h(t):=\|g(t)-g_J\|_X.
	\]
	We first show that \(h\) has scalar \(BMO\) seminorm controlled by the vector-valued \(BMO\) seminorm of \(g\).  For every interval \(I\subset J\), the scalar average \(h_I\) satisfies
	\[
	\frac1{|I|}\int_I |h(t)-h_I|\,dt
	\le
	\frac2{|I|}\int_I |h(t)-\|g_I-g_J\|_X|\,dt.
	\]
	By the reverse triangle inequality,
	\[
	\bigl|\|g(t)-g_J\|_X-\|g_I-g_J\|_X\bigr|
	\le
	\|g(t)-g_I\|_X.
	\]
	Therefore
	\[
	[h]_{BMO(0,T)}\le 2[g]_{BMO(0,T;X)}.
	\]
	Also
	\[
	h_J=\frac1T\int_0^T\|g(t)-g_J\|_X\,dt
	\le [g]_{BMO(0,T;X)}.
	\]
	The scalar John--Nirenberg inequality~\cite{JohnNirenberg1961} on the finite interval \(J\) gives
	\[
	\|h\|_{L^q(0,T)}
	\le
	C_q T^{1/q}\bigl(h_J+[h]_{BMO(0,T)}\bigr)
	\le
	C_q T^{1/q}[g]_{BMO(0,T;X)}.
	\]
	Finally,
	\[
	\|g\|_{L^q(0,T;X)}
	\le
	\|g-g_J\|_{L^q(0,T;X)}+T^{1/q}\|g_J\|_X
	\le
	C_qT^{1/q}\|g\|_{BMO_*(0,T;X)}.
	\]
	\end{proof}

	\begin{theorem}\label{thm:bmo-lions-peetre-trace}
	Let \((X_0,X_1)\) be a Banach couple with \(X_1\hookrightarrow X_0\).  Assume
	\[
	u\in BMO_*(0,T;X_1),
	\qquad
	u'\in BMO_*(0,T;X_0),
	\]
	where \(u'\) is the \(X_0\)-valued distributional derivative.  Then for every \(1<q<\infty\),
	\[
	u\in C\bigl([0,T];(X_0,X_1)_{1-1/q,q}\bigr),
	\]
	and
	\[
	\|u\|_{C([0,T];(X_0,X_1)_{1-1/q,q})}
	\le
	C_q\Bigl(
	\|u\|_{BMO_*(0,T;X_1)}+
	T\|u'\|_{BMO_*(0,T;X_0)}
	\Bigr),
	\]
	where $C_q$ is independent of $T$.
	\end{theorem}

	\begin{proof}
	By Lemma~\ref{lem:bmo-to-lq},
	\[
	u\in L^q(0,T;X_1),
	\qquad
	u'\in L^q(0,T;X_0).
	\]
	Hence
	\[
	u\in W^{1,q}(0,T;X_0)\cap L^q(0,T;X_1).
	\]
	Rescaling the classical trace theorem from $(0,1)$ to $(0,T)$ gives
	\[
	\|u\|_{C([0,T];(X_0,X_1)_{1-1/q,q})}
	\le C_q\Bigl(
	T^{-1/q}\|u\|_{L^q(0,T;X_1)}
	+T^{1-1/q}\|u'\|_{L^q(0,T;X_0)}
	\Bigr).
	\]
	Applying Lemma~\ref{lem:bmo-to-lq} to the two terms gives
	\[
	T^{-1/q}\|u\|_{L^q(X_1)}\le C_q\|u\|_{BMO_*(X_1)},
	\qquad
	T^{1-1/q}\|u'\|_{L^q(X_0)}\le C_qT\|u'\|_{BMO_*(X_0)},
	\]
	which is the asserted estimate.
	\end{proof}

	\begin{theorem}\label{thm:bmo-log-trace}
	Let \((X_0,X_1)\) be a Banach couple with \(X_1\hookrightarrow X_0\).  Assume
	\[
	u\in L^\infty(0,T;X_1),
	\qquad
	u'\in BMO_*(0,T;X_0).
	\]
	Then \(u\) has an \(X_0\)-valued continuous representative and, for every \(0<\eta<1\),
	\[
	u\in C\bigl([0,T];(X_0,X_1)_{\eta,\infty}\bigr).
	\]
	More precisely, for \(0<|t-s|\le T\),
	\[
	\|u(t)-u(s)\|_{(X_0,X_1)_{\eta,\infty}}
	\le
	C_\eta
	\Bigl(|t-s|\log\frac{eT}{|t-s|}\Bigr)^{1-\eta}
	\|u'\|_{BMO_*(0,T;X_0)}^{1-\eta}
	\|u\|_{L^\infty(0,T;X_1)}^\eta.
	\]
	Consequently,
	\[
	u\in C^{0,\alpha}\bigl([0,T];(X_0,X_1)_{\eta,\infty}\bigr)
	\]
	for every \(0<\alpha<1-\eta\).
	\end{theorem}

	\begin{proof}
	We first prove the base-space logarithmic modulus.  Let \(v\in BMO_*(0,T;X_0)\).  We claim that for every interval \(I\subset(0,T)\) of length \(h\),
	\[
	\|v_I\|_{X_0}
	\le
	C\|v\|_{BMO_*(0,T;X_0)}\log\frac{eT}{h}.
	\]
	Indeed, let \(J=(0,T)\).  Enlarge \(I\) dyadically inside \(J\): choose intervals
	\[
	I=I_0\subset I_1\subset\cdots\subset I_N\subset J
	\]
	with \(|I_{k+1}|\le 2|I_k|\) and \(N\le C\log(eT/h)\), and with \(|I_N|\simeq T\).  For nested intervals \(I_k\subset I_{k+1}\),
	\[
	\|v_{I_k}-v_{I_{k+1}}\|_{X_0}
	\le
	\frac1{|I_k|}\int_{I_k}\|v(t)-v_{I_{k+1}}\|_{X_0}\,dt
	\le
	C[v]_{BMO(0,T;X_0)}.
	\]
	Telescoping gives
	\[
	\|v_I\|_{X_0}
	\le
	\|v_J\|_{X_0}+C N[v]_{BMO(0,T;X_0)}
	\le
	C\|v\|_{BMO_*(0,T;X_0)}\log\frac{eT}{h}.
	\]

	Apply this to \(v=u'\).  If \(I=(s,t)\), then
	\[
	u(t)-u(s)=\int_s^t u'(\tau)\,d\tau=|I|(u')_I
	\]
	in \(X_0\).  Hence
	\[
	\|u(t)-u(s)\|_{X_0}
	\le
	C |t-s|\log\frac{eT}{|t-s|}
	\|u'\|_{BMO_*(0,T;X_0)}.
	\]
	On the other hand, for almost every \(s,t\),
	\[
	\|u(t)-u(s)\|_{X_1}
	\le 2\|u\|_{L^\infty(0,T;X_1)}.
	\]
	Using the elementary interpolation inequality
	\[
	\|x\|_{(X_0,X_1)_{\eta,\infty}}
	\le
	C_\eta\|x\|_{X_0}^{1-\eta}\|x\|_{X_1}^{\eta},
	\qquad x\in X_1,
	\]
	with \(x=u(t)-u(s)\), we obtain the stated logarithmic modulus at Lebesgue times.  The representative on all of \([0,T]\) is obtained by completion in \((X_0,X_1)_{\eta,\infty}\), using the modulus estimate and the \(X_0\)-continuous representative.  Finally, since
	\[
	h\log(eT/h)\le C_{\alpha,\eta,T}h^{\alpha/(1-\eta)}
	\]
	for every \(0<h\le T\) and every \(0<\alpha<1-\eta\), the H\"older conclusion follows.
	\end{proof}

	\begin{proposition}\label{prop:bmo-log-sharp}
	The base-space estimate used in Theorem~\ref{thm:bmo-log-trace}, with modulus $h\log(eT/h)$ in $X_0$, cannot in general be improved to a Lipschitz estimate with modulus $h$.
	\end{proposition}

	\begin{proof}
	It is enough to consider the scalar case \(X_0=X_1=\mathbb R\).  On \((0,1)\), let
	\[
	v(t)=\log\frac{e}{t}.
	\]
	The function \(v\) belongs to scalar \(BMO(0,1)\).  This follows directly from the scale invariance of the logarithm near the origin: on intervals \((0,r)\), the mean oscillation is independent of \(r\), and on intervals away from the origin the logarithm has bounded mean oscillation by elementary monotonicity estimates.
	Define
	\[
	u(t)=\int_0^t v(s)\,ds.
	\]
	Then
	\[
	u(t)=t\log\frac{e}{t}+t
		\simeq t\log\frac{e}{t}
	\]
	as \(t\downarrow0\).  Hence \(u'\in BMO(0,1)\), but \(u\) does not satisfy a Lipschitz estimate at the origin.  More precisely,
	\[
	\frac{|u(t)-u(0)|}{t}=\log\frac{e}{t}+1\to\infty.
	\]
	Thus an estimate assuming only \(u'\in BMO\) cannot replace the base-space modulus \(t\log(e/t)\) by \(Ct\). This scalar example alone does not establish the sharpness of the interpolated modulus for each \(\eta>0\).
	\end{proof}

\section{A nonlinear parabolic application: exact Hamilton--Jacobi traces}\label{sec:hj-applications}
The weak interpolation fine index of the abstract trace theorem is realized by a
nonlinear parabolic equation. Throughout this section, $\mathbb T^d=(\mathbb R/2\pi\mathbb Z)^d$ is the flat
torus, $\nu>0$, $T>0$, $1<p<\infty$, $0<\lambda<1$, and
\[
 \theta=1-\frac{1-\lambda}{p},\qquad s=2\theta\in(0,2).
\]
We use inhomogeneous Sobolev and Besov spaces. On the torus, the standard
real-interpolation identity reads
\begin{equation}\label{eq:besov-interpolation}
 (L^2(\mathbb T^d),H^2(\mathbb T^d))_{\theta,\infty}
  =B_{2,\infty}^{2\theta}(\mathbb T^d)
\end{equation}
with equivalent norms. The notation $\mathcal M_t^{p,\lambda}(Y)$ means
$\mathcal M^{p,\lambda}(0,T;Y)$.

\paragraph{A composition estimate used below.}
For real $z\in H^2(\mathbb T^d)\cap L^\infty$,
\begin{equation}\label{eq:GN-L4}
 \|\nabla z\|_4^2\le C_d\|z\|_\infty\|D^2z\|_2.
\end{equation}
Indeed, for smooth $z$, periodic integration by parts gives
\[
 \int |\nabla z|^4
 =-\int z\,\operatorname{div}(|\nabla z|^2\nabla z)
 \le (2+\sqrt d)\|z\|_\infty\|D^2z\|_2\|\nabla z\|_4^2.
\]
Division proves the estimate; heat regularization and Fatou's lemma extend
it to the stated class. The estimate is valid in every dimension.

For completeness, this also proves the composition fact needed at the weak
Besov endpoint. Suppose that $z\in L^\infty\cap B_{2,\infty}^{2\theta}$,
$0<\theta<1$, and that $F$ is $C^2$ on a neighborhood of the essential
range of $z$ and its convex hull. Let $z_r=e^{r\Delta}z$. For $0<r\le1$,
spectral estimates and the definition of $K$ give
\[
 \|z-z_r\|_2+r\|z_r\|_{H^2}
 \le C K(r,z;L^2,H^2),\qquad \|z_r\|_\infty\le\|z\|_\infty.
\]
The chain rule and \eqref{eq:GN-L4} imply
$\|F(z_r)\|_{H^2}\le C_{F,\|z\|_\infty}(1+\|z_r\|_{H^2})$,
while $\|F(z)-F(z_r)\|_2\le C_F\|z-z_r\|_2$. Hence
\[
 K(r,F(z);L^2,H^2)
 \le C_{F,\|z\|_\infty}\bigl(K(r,z;L^2,H^2)+r\bigr).
\]
Taking the $r^{-\theta}$ supremum for $0<r\le1$ and using the $L^2$ norm
for $r>1$ proves $F(z)\in B_{2,\infty}^{2\theta}$. Applying this argument
to the exponential, and then to the logarithm on a compact positive
interval, proves both directions required in the following theorem.

\begin{theorem}[Exact nonlinear Morrey initial-data criterion]\label{thm:hj-morrey-exact}
Let $u_0$ be real valued and bounded on $\mathbb T^d$, and define
\begin{equation}\label{eq:hopf-cole-solution}
 w(t)=e^{\nu t\Delta}\exp\!\left(-\frac{u_0}{2\nu}\right),
 \qquad u(t)=-2\nu\log w(t).
\end{equation}
Then $u$ is the bounded Hopf--Cole solution of
\begin{equation}\label{eq:hj-full}
 \begin{cases}
 u_t-\nu\Delta u+\frac12|\nabla u|^2=0,
   & (t,x)\in(0,T)\times\mathbb T^d,\\
 u(0,x)=u_0(x),&x\in\mathbb T^d,
 \end{cases}
\end{equation}
and the following statements are equivalent:
\begin{enumerate}[label=\textup{(\roman*)}]
\item $u_0\in B_{2,\infty}^{s}(\mathbb T^d)$;
\item $u\in\mathcal M_t^{p,\lambda}(H_x^2)$ and
$u_t\in\mathcal M_t^{p,\lambda}(L_x^2)$.
\end{enumerate}
In particular, $B_{2,\infty}^{s}\cap L^\infty$ is the exact bounded
initial-data class for this Morrey regularity of the Hopf--Cole solution.
\end{theorem}

\begin{proof}
Put $w_0=\exp(-u_0/(2\nu))$. As $u_0$ is bounded, there are constants
$0<c_0\le C_0$ such that $c_0\le w_0\le C_0$ almost everywhere.
The positivity-preserving heat semigroup gives
\begin{equation}\label{eq:hopf-positive}
 c_0\le w(t,x)\le C_0,\qquad 0<t<T.
\end{equation}
The identities
\[
 w_t-\nu\Delta w
 =-\frac{w}{2\nu}\bigl(u_t-\nu\Delta u+\tfrac12|\nabla u|^2\bigr)
\]
and $u=-2\nu\log w$ establish the equation for positive times and the
initial trace in $L^2$. More explicitly, heat strong continuity gives
$w(t)\to w_0$ in $L^2$, and the logarithm is Lipschitz on $[c_0,C_0]$, so
$u(t)\to u_0$ in $L^2$. The composition estimate just proved gives
\begin{equation}\label{eq:hj-composition-besov}
 u_0\in B_{2,\infty}^s\quad\Longleftrightarrow\quad
 w_0\in B_{2,\infty}^s.
\end{equation}

Assume (i). Let $H=L^2(\mathbb T^d)$ and $A=I-\nu\Delta$ with
$D(A)=H^2(\mathbb T^d)$. For any decomposition $w_0=a+b$,
$a\in H$, $b\in D(A)$, the spectral estimate for the analytic heat
semigroup gives, for $0<t\le T$,
\[
 \|Ae^{\nu t\Delta}w_0\|_H
 \le C_T\bigl(t^{-1}\|a\|_H+\|b\|_{D(A)}\bigr).
\]
Taking the infimum over decompositions and using
\eqref{eq:besov-interpolation} yields
\begin{equation}\label{eq:heat-K}
 \|Aw(t)\|_H\le C_Tt^{-1}K(t,w_0;H,D(A))
 \le C_Tt^{\theta-1}\|w_0\|_{B_{2,\infty}^{2\theta}}.
\end{equation}
Also $\|w(t)\|_H\le\|w_0\|_H$, while
$w_t=\nu\Delta w=(I-A)w$. Hence both $w_t$ and $w$ in $H^2$
are bounded by a constant multiple of $1+t^{\theta-1}$ in their
respective spatial norms.

The scalar bound is compatible with exactly the Morrey normalization:
\begin{equation}\label{eq:power-morrey}
 \sup_{I\subset(0,T)}|I|^{-\lambda/p}
 \left(\int_I t^{-(1-\lambda)}\,dt\right)^{1/p}<\infty
 \quad(0<\lambda<1).
\end{equation}
Indeed, writing $I=(a,a+h)$, the integral is $O(h^\lambda)$ if
$a\le h$, while for $a>h$ it is at most
$h a^{-(1-\lambda)}\le h^\lambda$.
Thus $w\in\mathcal M_t^{p,\lambda}(H_x^2)$ and
$w_t\in\mathcal M_t^{p,\lambda}(L_x^2)$.

To return to $u$, differentiate $u=-2\nu\log w$:
\[
 u_t=-2\nu\frac{w_t}{w},\qquad
 D^2u=-2\nu\left(\frac{D^2w}{w}-
       \frac{\nabla w\otimes\nabla w}{w^2}\right).
\]
Estimate~\eqref{eq:GN-L4}, together with
\eqref{eq:hopf-positive}, implies
\[
 \|D^2u(t)\|_2+\|u_t(t)\|_2
 \le C(\nu,d,\|u_0\|_\infty)
   \bigl(\|D^2w(t)\|_2+\|w_t(t)\|_2\bigr).
\]
The lower-order terms of $\|u(t)\|_{H^2}$ are controlled by the
uniform $L^\infty$ bound on $u$ and standard periodic interpolation.
Therefore (ii) follows. All constants may depend on $T$, $\nu$, $d$, $p$, $\lambda$, and $\|u_0\|_\infty$; no smallness of $u_0$ is imposed and no uniformity as $\nu\downarrow0$ is asserted.

Conversely, if (ii) holds, Theorem~\ref{thm:exact-morrey-trace-space}
with $(X_0,X_1)=(L^2(\mathbb T^d),H^2(\mathbb T^d))$ gives
\[
 u_0=u(0)\in (L^2,H^2)_{\theta,\infty}
  =B_{2,\infty}^{s}(\mathbb T^d),
\]
proving (i).
\end{proof}

\begin{proposition}[The weak fine index is realized by nonlinear solutions]
\label{prop:hj-fine-sharp}
For each $0<\lambda<1$, $1<p<\infty$ and $0<\varepsilon<1$,
there exists bounded real initial data $u_0$ such that the solution of
\eqref{eq:hj-full} satisfies conclusion \textup{(ii)} of
Theorem~\ref{thm:hj-morrey-exact}, although
\[
 u_0\notin B_{2,q}^{s}(\mathbb T^d)
 \qquad\text{for every }1\le q<\infty.
\]
\end{proposition}
\begin{proof}
Use the lacunary series
\[
 u_0(x)=\varepsilon\sum_{k=1}^\infty 2^{-ks}\cos(2^k x_1).
\]
Because $s>0$, it converges uniformly and $u_0\in L^\infty$.
By the dyadic Fourier-block characterization of periodic Besov spaces,
$\sup_k2^{ks}\|\Delta_k u_0\|_2<\infty$ but the sequence
$\bigl(2^{ks}\|\Delta_k u_0\|_2\bigr)_k$ is not in any
$\ell^q$, $q<\infty$ (one may pass to a fixed-step lacunary
subsequence to avoid overlap among cutoff supports).
Thus $u_0\in B_{2,\infty}^{s}$ and $u_0\notin B_{2,q}^{s}$ for finite $q$. The same blocks also show $u_0\notin B_{2,\infty}^{s+\delta}$ for every $\delta>0$.
Apply Theorem~\ref{thm:hj-morrey-exact}.
\end{proof}

\subsection{A complementary time-BMO estimate for a forced equation}
For the bounded unforced solutions considered above, the hypothesis
$u\in L^\infty(0,T;H^2)$ already implies $u_t\in L^\infty(0,T;L^2)$:
use the equation and \eqref{eq:GN-L4} with $z=u$. Thus in that setting
Theorem~\ref{thm:endpoint-subcritical-trace} gives the stronger modulus
$|t-s|^{1-\eta}$, without a logarithm. A genuine time-BMO application is
instead supplied by a forcing term:
\begin{equation}\label{eq:hj-forced}
 \begin{cases}
 u_t-\nu\Delta u+\tfrac12|\nabla u|^2=f(t,x),
       &(t,x)\in(0,T)\times\mathbb T^d,\\
 u(0,x)=u_0(x),&x\in\mathbb T^d.
 \end{cases}
\end{equation}
\begin{corollary}\label{cor:hj-bmo}
Let $u$ be a bounded distributional solution of \eqref{eq:hj-forced} with
\[
 U:=\|u\|_{L^\infty(0,T;H^2)}<\infty,
 \qquad M:=\|u\|_{L^\infty((0,T)\times\mathbb T^d)}<\infty,
 \qquad f\in BMO_*(0,T;L^2).
\]
Then $u_t\in BMO_*(0,T;L^2)$ and
\[
 \|u_t\|_{BMO_*(L^2)}\le C_d\bigl(\|f\|_{BMO_*(L^2)}+(\nu+M)U\bigr)=:\mathcal A.
\]
For every $0<\eta<1$, $u$ has a representative in
$C([0,T];B_{2,\infty}^{2\eta})$, and
\[
 \|u(t)-u(s)\|_{B_{2,\infty}^{2\eta}}
 \le C_\eta\left(|t-s|\log\frac{eT}{|t-s|}\right)^{1-\eta}
       \mathcal A^{1-\eta}U^\eta
 \qquad(0<|t-s|\le T).
\]
\end{corollary}
\begin{proof}
Estimate~\eqref{eq:GN-L4} gives
$\||\nabla u(t)|^2\|_2\le C_dM\|D^2u(t)\|_2$.
Consequently $\nu\Delta u-\tfrac12|\nabla u|^2$ belongs to
$L^\infty(0,T;L^2)$ and hence to $BMO_*(0,T;L^2)$.
The equation gives the stated estimate for $u_t$.
Theorem~\ref{thm:bmo-log-trace} and
$(L^2,H^2)_{\eta,\infty}=B_{2,\infty}^{2\eta}$ finish the proof.
\end{proof}
The estimate is conditional on the displayed spatial bounds; it is not an
existence theorem for arbitrary forcing. Nevertheless the base-space
logarithm is attained by actual solutions. For example,
\[
 f(t,x)=\log(eT/t),\qquad
 u(t,x)=t\log(eT/t)+t,\qquad u_0=0
\]
is a bounded, spatially constant solution, with $u\in L^\infty(H^2)$ and
$u_t\in BMO_*(L^2)\setminus L^\infty(L^2)$. Its $L^2$ modulus at zero
has order $t\log(eT/t)$. No sharpness of every interpolated modulus is
asserted. In particular, no invariance of time-BMO under the Hopf--Cole
change of variables is used.

\subsection{Initial and terminal traces in a mean-field-game system}
A finite-dimensional forward--backward model has the form
\begin{equation}\label{eq:mfg-application}
\begin{cases}
-v_t-\nu\Delta v+\frac12|\nabla v|^2=F(x,\rho),
 &v(T,x)=G(x,\rho(T)),\\
\rho_t-\nu\Delta\rho-\operatorname{div}(\rho\nabla v)=0,
 &\rho(0,x)=\rho_0(x).
\end{cases}
\end{equation}
If, for some solution whose density $\rho$ and value function $v$ are
$L^2(\mathbb T^d)$-valued,
\[
 \rho,v\in\mathcal M_t^{p,\lambda}(H_x^2),\qquad
 \rho_t,v_t\in\mathcal M_t^{p,\lambda}(L_x^2),
\]
then time reversal for $v$ and Theorem~\ref{thm:exact-morrey-trace-space} give the compatible
initial and terminal traces
\[
 \rho_0\in B_{2,\infty}^{s}(\mathbb T^d),\qquad
 v(T)\in B_{2,\infty}^{s}(\mathbb T^d).
\]
This observation does not establish maximal regularity or existence for the
coupled system. In particular it must not be conflated with the
infinite-dimensional measure-valued mean-field-game model studied in
\cite{FedericoGozziSwiech2026}, where the Fokker--Planck unknown is a
probability measure and not an $L^2$ density.

\section{Concluding remarks}
Time-Morrey regularity determines an exact weak interpolation trace space: the smoothness exponent $\theta=1-(1-\lambda)/p$ and the fine index $\infty$ are both optimal. The constructive extension realizes every trace, while the pointwise estimates distinguish bounded trace values from continuity in the endpoint interpolation norm. The BMO theorems provide a complementary family of continuous finite-index traces and a quantitative logarithmic modulus. For the quadratic viscous Hamilton--Jacobi equation, the abstract characterization becomes an exact Besov criterion for nonlinear evolution regularity, with sharpness attained by bounded lacunary data. Forced solutions exhibit the base-space logarithmic behavior, and forward--backward systems provide a simultaneous application to initial and terminal compatibility.

\section*{Statements and Declarations}
\noindent\textbf{Competing interests.} The authors declare no competing interests.

\noindent\textbf{Data availability.} No datasets were generated or analysed during this study.

\noindent\textbf{Use of generative AI.} Gemini was used in the preparation and revision of this manuscript for exposition, literature searches, and assistance with language checking. Responsibility for all statements, proofs, and references rests with the authors.

\end{document}